\documentclass[a4paper, 11pt]{article}
\title{Cayley Graph on Symmetric Group Generated by Elements Fixing $k$ Points}
\author{Cheng Yeaw Ku \thanks{ Department of Mathematics, National University of
Singapore, Singapore 117543. E-mail: matkcy@nus.edu.sg} \and Terry Lau \thanks{
Institute of Mathematical Sciences, University of Malaya, 50603
Kuala Lumpur, Malaysia. \newline E-mail:
terrylsc@hotmail.com} \and Kok Bin Wong \thanks{
Institute of Mathematical Sciences, University of Malaya, 50603
Kuala Lumpur, Malaysia. \newline E-mail:
kbwong@um.edu.my}}
\date{\today}

\usepackage{amsmath, amssymb, amsthm, indentfirst, bm, verbatim, yfonts, bbm, dsfont}

\theoremstyle{plain}
\newtheorem{thm}{Theorem}[section]
\newtheorem{cor}[thm]{Corollary}

\newtheorem{lem}[thm]{Lemma}
\newtheorem{eg}[thm]{Example}

\theoremstyle{definition}

\theoremstyle{remark}

\newcommand{\w}{\widehat}
\newcommand{\sgn}{\textnormal{sign}}

\begin{document}
\maketitle
\markboth{Cayley Graph on Symmetric Group Generated by Elements Fixing $k$ Points}
{Cayley Graph on Symmetric Group Generated by Elements Fixing $k$ Points}
\renewcommand{\sectionmark}[1]{}

\begin{abstract} \noindent
Let $\mathcal{S}_{n}$ be the symmetric group on $[n]=\{1, \ldots, n\}$. The $k$-point fixing graph $\mathcal{F}(n,k)$ is defined to be the graph with vertex set $\mathcal{S}_{n}$ and two vertices $g$, $h$ of $\mathcal{F}(n,k)$ are joined if and only if  $gh^{-1}$ fixes exactly $k$ points. In this paper, we derive a recurrence formula for the eigenvalues of $\mathcal{F}(n,k)$. Then we apply our result to determine the sign of the eigenvalues of $\mathcal{F}(n,1)$.
\end{abstract}

\vspace{5mm}
\noindent
{\sc keywords:} Arrangement graph, Cayley graphs, symmetric group, Alternating Sign Property. \\[.1cm]

\section{Introduction}

Let $G$ be a finite group and $S$ be an inverse closed subset of $G$, i.e., $1 \notin S$ and $s \in S \Rightarrow s^{-1} \in G$. The \emph{Cayley graph} $\Gamma(G,S)$ is the graph which has the elements of $G$ as its vertices and two vertices $u, v \in G$ are joined by an edge if and only if $v=su$ for some $s\in S$.

 A Cayley graph $\Gamma(G,S)$ is said to be {\em normal} if $S$ is closed under conjugation. It is well known that the eigenvalues of a normal Cayley graph $\Gamma(G,S)$ can be expressed in terms of the irreducible characters of $G$.

\begin{thm}[\cite{babai, DS, Lub, Ram}]\label{frobenius-schur}
The eigenvalues of a normal Cayley graph $\Gamma(G,S)$ are given by
\begin{eqnarray}
\eta_{\chi} & = & \frac{1}{\chi(1)} \sum_{s \in S} \chi(s),\notag
\end{eqnarray}
where $\chi$ ranges over all the irreducible characters of $G$. Moreover, the multiplicity of $\eta_{\chi}$ is $\chi(1)^{2}$.
\end{thm}

Let $\mathcal{S}_{n}$ be the symmetric group on $[n]=\{1, \ldots, n\}$ and $S\subseteq \mathcal S_n$ be closed under conjugation. Since central
characters are algebraic integers (\cite[Theorem 3.7 on p. 36]{Isaacs}) and that the characters of the symmetric group are integers (\cite[2.12 on p. 31]{Isaacs} or \cite[Corollary 2 on p. 103]{Serre}), by Theorem \ref{frobenius-schur}, the eigenvalues of $\Gamma(\mathcal S_n,S)$ are integers.

\begin{cor}\label{cor_cayley_integer}
The eigenvalues of a normal Cayley graph $\Gamma(\mathcal S_n,S)$ are integers.
\end{cor}

For $k\le n$, a {\em $k$-permutation} of $[n]$ is an injective function from $[k]$ to $[n]$.
So any $k$-permutation $\pi$ can be represented by a vector $(i_1,\ldots,i_k)$ where $\pi(j)=i_j$ for $j=1,\ldots,k$. Let $1\leq r\leq k\leq n$.
The {\em $(n,k,r)$-arrangement graph}   $A(n,k,r)$ has all the $k$-permutations of $[n]$ as vertices and two $k$-permutations are adjacent if they differ in exactly $r$ positions. Formally, the vertex set $V(n,k)$ and edge set $E(n,k,r)$ of $A(n,k,r)$ are
\begin{align}
V(n,k) & =\big\{(p_1,p_2,\dots ,p_k) \mid p_i\in [n]\ \textnormal{and}\ p_i\neq p_j\ \textnormal{for}\ i\neq j\big\},\notag\\
E(n,k,r) &=\big\{ \{(p_1,p_2,\dots ,p_k),(q_1,q_2,\dots ,q_k)\} \mid p_i\neq q_i \ \textnormal {for $i\in R$ and}\notag\\
&\hskip 1cm p_j= q_j \ \textnormal{for all $j\in [k]\setminus R$ for some $R\subseteq [k]$ with $\vert R\vert=r$}\big\}.\notag
\end{align}
Note that $\vert V(n,k)\vert=n!/(n-k)!$ and $A(n,k,r)$ is a regular graph \cite[Theorem 4.2]{Chen2}. In particular, $A(n,k,1)$ is a $k(n-k)$-regular graph. We note here that $A(n,k,1)$ is called the partial permutation graph in \cite{Kra}.

The family of the arrangement graphs $A(n,k,1)$ was first introduced in \cite{dt} as an interconnection network model for parallel computation.
In the interconnection network model, each processor has its own memory unit and
communicates with the other processors through a topological network, i.e. a graph. Many properties of the arrangement graphs $A(n,k,1)$ have been studied in \cite{cgq,cll,cly,cqs,cc,dt,ttt,zx}.

The eigenvalues of the arrangement graphs $A(n,k,1)$ were first studied in \cite{Chen} by using a method developed by Godsil and McKay  \cite{gm}. A relation between the eigenvalues of $A(n,k,r)$ and certain Cayley graphs was given in \cite{Chen2}.

The {\em derangement graph} $\Gamma_{n}$ is the Cayley graph $\Gamma(\mathcal{S}_{n}, D_n)$ where $D_n$ is the set of derangements in $\mathcal{S}_{n}$. That is, two vertices $g$, $h$ of $\Gamma_{n}$ are joined if and only if $g(i) \not = h(i)$ for all $i \in [n]$, or equivalently $gh^{-1}$ fixes no point.
Since $D_n$ is closed under conjugation, by Corollary \ref{cor_cayley_integer}, the eigenvalues of the derangement graph are integers.  The lower and upper bounds of the absolute values of these eigenvalues have been studied in \cite{ku-wales, ku-wong, renteln}. Note that the derangement graph is a kind of arrangement graph, i.e., $\Gamma_{n}=A(n,n,n)$.

Let $0\leq k < n$ and $\mathcal{S} (n,k)$ be the set of all $\sigma \in \mathcal{S}_{n}$ such that $\sigma$ fixes exactly $k$ elements. Note that $\mathcal{S} (n,k)$ is an inverse closed subset of $\mathcal{S}_{n}$. The $k$-\emph{point fixing graph} is defined to be
  \[ \mathcal{F}(n,k) = \Gamma (\mathcal S_n, \mathcal{S} (n,k)). \]
  That is, two vertices $g$, $h$ of $\mathcal{F}(n,k)$ are joined if and only if  $gh^{-1}$ fixes exactly $k$ points.
Note that the $k$-point fixing graph is also a kind of arrangement graph, i.e., $\mathcal{F}(n,k)=A(n,n,n-k)$. Furthermore, the 0-point fixing graph is the derangement graph, i.e., $\mathcal{F}(n,0)=\Gamma_{n}=A(n,n,n)$.

Clearly, $\mathcal{F}(n,k)$ is vertex-transitive, so it is $|\mathcal{S}(n,k)|$-regular and  the largest eigenvalue of $\mathcal{F}(n,k)$ is $|\mathcal{S}(n,k)|$. Furthermore, $\mathcal{S} (n,k)$ is closed under conjugation. Therefore, by Corollary \ref{cor_cayley_integer}, the eigenvalues of the $k$-point fixing graph are integers.  However, the complete set of spectrum of $\mathcal{F}(n,k)$ is not known. The purpose of this paper is to study the eigenvalues of $\mathcal{F}(n,k)$.

Recall that a partition $\lambda$ of $n$, denoted by $\lambda \vdash n$, is a weakly decreasing sequence $\lambda_{1} \ge \ldots \ge \lambda_{r}$ with $\lambda_{r} \ge 1$ such that $\lambda_{1} + \cdots + \lambda_{r} = n$. We write $\lambda = (\lambda_{1}, \ldots, \lambda_{r})$. The {\em size} of $\lambda$, denoted by $|\lambda|$, is $n$ and each $\lambda_{i}$ is called the {\em $i$-th part} of the partition. We also use the notation $(\mu_{1}^{a_{1}}, \ldots, \mu_{s}^{a_{s}}) \vdash n$ to denote the partition where $\mu_{i}$ are the distinct nonzero parts that occur with multiplicity $a_{i}$. For example,
\[ (5,5,4,4,2,2,2,1) \longleftrightarrow (5^{2}, 4^{2}, 2^{3}, 1). \]

It is well known that both the conjugacy classes of $\mathcal{S}_{n}$ and the irreducible characters of $\mathcal{S}_{n}$ are indexed by partitions $\lambda$ of $[n]$. Since $\mathcal{S} (n,k)$ is closed under conjugation, the eigenvalue $\eta_{\chi_{\lambda}}(k)$ of the $k$-point fixing graph can be denoted by $\eta_{\lambda}(k)$. Throughout the paper, we shall use this notation.

The paper is organized as follows. In Section 2, we provide some known results regarding the eigenvalues of  $\mathcal{F} (n,0)$. In Section 3, we prove a recurrence formula for the eigenvalues of $\mathcal{F}(n,k)$ (Theorem \ref{eigen-n1}). In Section 4, we prove 
some inequalities for the eigenvalues of $\mathcal{F}(n,0)$ which will be used to prove the following Alternating Sign Property (ASP) for $\mathcal{F}(n,1)$:

\begin{thm}\label{ASP_1}\textnormal{(ASP for $\mathcal{F}(n,1)$)} Let $n\geq 2$ and $\lambda = (\lambda_{1}, \ldots, \lambda_{r}) \vdash n$.
\begin{itemize}
\item[\textnormal{(a)}] $\eta_\lambda (1)=0$ if and only if $\lambda= (n-1,1)$ or $\lambda=(2,1^{n-2})$.
\item[\textnormal{(b)}] If $r=1$ and $\lambda\neq (2)$, then $\eta_\lambda (1)>0$.
\item[\textnormal{(c)}] If $r\geq 2$ and $\lambda\neq (n-1,1)$ or $(2,1^{n-2})$, then
\begin{eqnarray}
\sgn(\eta_{\lambda}(1)) & = & (-1)^{|\lambda|-\lambda_{1}-1} \nonumber \\
& = & (-1)^{(\# \textnormal{cells under the first row of $\lambda$})-1}\notag
\end{eqnarray}
where $\sgn(\eta_\lambda(1))$ is $1$ if $\eta_{\lambda}(1)$ is positive or $-1$ if $\eta_{\lambda}(1)$ is negative.
\end{itemize}
\end{thm}

In Section 6,  we provide a list of eigenvalues of  $\mathcal{F}(n,1)$ for small $n$.

\section{Known results for eigenvalues of $\mathcal{F}(n,0)$}

To describe the Renteln's recurrence formula for $\mathcal{F}(n,0)$, we require some terminology. To the Ferrers diagram of a partition $\lambda$, we assign
$xy$-coordinates to each of its boxes by defining the
upper-left-most box to be $(1,1)$, with the $x$ axis increasing to
the right and the $y$ axis increasing downwards. Then the {\em hook}
of $\lambda$ is the union of the boxes $(x',1)$ and $(1, y')$ of the
Ferrers diagram of $\lambda$, where $x' \ge 1$, $y' \ge 1$. Let
$\w{h}_{\lambda}$ denote the hook of $\lambda$ and let $h_{\lambda}$
denote the size of $\w{h}_{\lambda}$. Similarly, let
$\w{c}_{\lambda}$ and $c_{\lambda}$ denote the first column of
$\lambda$ and the size of $\w{c}_{\lambda}$ respectively. Note that
$c_{\lambda}$ is equal to the number of rows of $\lambda$. When
$\lambda$ is clear from the context, we will replace $\w{h}_{\lambda}$,
$h_{\lambda}$, $\w{c}_{\lambda}$ and $c_{\lambda}$ by $\w{h}$, $h$,
$\w{c}$ and $c$ respectively. Let $\lambda-\w{h} \vdash n-h$ denote
the partition obtained from $\lambda$ by removing its hook. Also,
let $\lambda-\w{c}$ denote the partition obtained from $\lambda$ by
removing the first column of its Ferrers diagram, i.e.
$(\lambda_{1}, \ldots, \lambda_{r})-\w{c} = (\lambda_{1}-1, \ldots,
\lambda_{r}-1) \vdash n-r$.

\begin{thm}\label{rentelnformula}\textnormal{(\cite[Theorem 6.5]{renteln}  Renteln's Recurrence Formula)}
  For any partition $\lambda=(\lambda_1,\ldots,\lambda_r) \vdash n$, the eigenvalues of the derangement graph $\mathcal{F}(n,0)$ satisfy the following recurrence:
  \[ \eta_\lambda (0) = (-1)^h \eta_{\lambda - \w{h}} (0) + (-1)^{h+\lambda_1} h \eta_{\lambda-\w{c}} (0) \]
  with initial condition $\eta_\emptyset(0) = 1$.
\end{thm}

To describe the Ku-Wong's recurrence formula for $\mathcal{F}(n,0)$, we need a new terminology.  For a partition $\lambda = (\lambda_1, \ldots, \lambda_r) \vdash n$, let $\w{l}_\lambda$ denote the last row of $\lambda$ and $l_\lambda$ denote the size of $\w{l}_\lambda$. Clearly, we have $l_{\lambda} = \lambda_r$. Let $\lambda - \w{l}_\lambda$ denote the partition obtained from $\lambda$ by deleting the last row. When $\lambda$ is clear from the context, we will replace $\w{l}_\lambda$, $l_\lambda$ by $\w{l}$ and $l$ respectively.

\begin{thm}\label{ku-wongformula} \textnormal{(\cite[Theorem 1.4]{ku-wong} Ku-Wong's Recurrence Formula)}
  For any partition $\lambda =(\lambda_1,\ldots,\lambda_r) \vdash n$, the eigenvalues of the derangement graph $\mathcal{F}(n,0)$ satisfy the following recurrence:
  \[ \eta_\lambda (0) = (-1)^{\lambda_r} \eta_{\lambda - \w{l}} (0) + (-1)^{r-1} \lambda_r \eta_{\lambda-\w{c}} (0) \]
  with initial condition $\eta_\emptyset (0)= 1$.
\end{thm}

The following theorem is called the  Alternating Sign Property (ASP) for $\mathcal{F}(n,0)$. 

\begin{thm}\label{ASP_0} \textnormal{(\cite[Theorem 1.2]{ku-wales},\cite[Theorem 1.3]{ku-wong})} Let $n\geq 2$. For any partition $\lambda = (\lambda_{1}, \ldots, \lambda_{r}) \vdash n$,
\begin{eqnarray}
\sgn(\eta_{\lambda}(0)) & = & (-1)^{|\lambda|-\lambda_{1}} \nonumber \\
& = & (-1)^{\# \textnormal{cells under the first row of $\lambda$}}\notag
\end{eqnarray}
where $\sgn(\eta_\lambda(0))$ is $1$ if $\eta_{\lambda}(0)$ is positive or $-1$ if $\eta_{\lambda}(0)$ is negative.
\end{thm}

The following corollary is a consequence of Theorem \ref{ku-wongformula} and \ref{ASP_0}.

\begin{cor}\label{cor_ku-wongformula}   For any partition $\lambda =(\lambda_1,\ldots,\lambda_r) \vdash n$ with $r\geq 2$, the absolute value of the eigenvalues of the derangement graph $\mathcal{F}(n,0)$ satisfy the following recurrence:
  \[ \vert\eta_\lambda (0)\vert = \vert \eta_{\lambda - \w{l}} (0)\vert + \lambda_r\vert  \eta_{\lambda-\w{c}} (0)\vert \]
  with initial condition $\vert\eta_\emptyset (0)\vert= 1$.
\end{cor}

\section{Recurrence Formula}

For each $\sigma\in\mathcal S_n$, we denote it's conjugacy class by $\textnormal{Con}_{\mathcal S_n}(\sigma)$, i.e., $\textnormal{Con}_{\mathcal S_n}(\sigma)=\{\gamma^{-1}\sigma\gamma\ :\ \gamma\in \mathcal S_n\}$. Let $\mu\vdash n$ be the partition that represents $\textnormal{Con}_{\mathcal S_n}(\sigma)$. We shall denote the size of $\textnormal{Con}_{\mathcal S_n}(\sigma)$ by $N_{\mathcal S_n}(\mu)$.

Let $A\subseteq \mathcal{S}_{n}$ and $\alpha\in \mathcal{S}_{n}$. The set $\alpha^{-1}A\alpha$ is defined as
\begin{equation}
\alpha^{-1}A\alpha=\{ \alpha^{-1} \sigma\alpha\ :\ \sigma\in A\}.\notag
\end{equation}
Let $0\leq k<n$. Each $\beta\in \mathcal{S}_{n-k}$ can be considered as an element $\overline\beta$ of $\mathcal{S}_{n}$ by defining $\overline\beta(j)=\beta(j)$ for $1\leq j\leq n-k$ and $\overline\beta (j)=j$ for $n-k+1\leq j\leq n$. The $\overline \beta$ is called the \emph{extension} of $\beta$ to $\mathcal S_n$. 
The set of derangements $D_{n-k}$ in $\mathcal{S}_{n-k}$ can be considered as a subset of $\mathcal{S}_{n}$ ($\overline D_{n-k}=\{\overline\sigma\ :\ \sigma\in  D_{n-k}\}$). Furthermore, $\bigcup_{\sigma\in \mathcal{S}_{n}} \sigma^{-1}\overline D_{n-k}\sigma\subseteq \mathcal{S}(n,k)$.

Let $\gamma \in \mathcal{S}(n,k)$. Then $\gamma $ fixes exactly $k$ elements, i.e., $\gamma(i_j)=i_j$ for $j=1,2,\dots, k$ and $\gamma (a)\neq a$ for $a\in [n]\setminus \{i_1,i_2,\dots, i_k\}=\{b_1,b_2,\dots, b_{n-k}\}$. Let $\sigma_0(b_j)=j$ for $1\leq j\leq n-k$ and $\sigma_0(i_j)=n-k+j$ for $1\leq j\leq k$. Then  $\sigma_0\in \mathcal{S}_{n}$ and $\sigma_{0}^{-1}\gamma \sigma_{0}\in \overline D_{n-k}$. Hence, the following lemma follows.

\begin{lem}\label{structure_of_S(n,k)} \
\begin{equation}
\mathcal{S}(n,k)=\bigcup_{\sigma\in \mathcal{S}_{n}} \sigma^{-1}\overline D_{n-k}\sigma.\notag
\end{equation}
\end{lem}

Let $\lambda \vdash n$. For a box with coordinate $(a,b)$ in the Ferrers diagram of $\lambda$, the \emph{hook-length} $h_{\lambda}(a,b)$ is the size of the set of all the boxes with coordinate $(i,j)$ where $i=a$ and $j\geq b$, or $i\geq a$ and $j=b$.  The following lemma is well-known \cite[4.12 on p. 50]{FH}.

\begin{lem}\label{dimension_f} \
\begin{equation}
\chi_{\lambda}(1)=\frac{n!}{\prod h_{\lambda}(a,b)},\notag
\end{equation}
where the product is over all the boxes $(a,b)$ in the Ferrers diagram of $\lambda$.
\end{lem}

For convenience, let us denote the dimension of $\lambda$ by $f^{\lambda}$, i.e., $f^{\lambda}=\chi_{\lambda}(1)$. By Lemma \ref{structure_of_S(n,k)}, there are $\sigma_{k1}$, $\sigma_{k2}$, $\dots$, $\sigma_{ks_k}\in  D_{n-k}$ such that
\begin{equation}
\mathcal{S}(n,k)=\bigcup_{i=1}^{s_k} \textnormal{Con}_{\mathcal S_n}(\overline \sigma_{ki}),\label{reduction_conjugate}
\end{equation}
and $\sigma_{ki}$ is not conjugate to $\sigma_{kj}$ in $\mathcal S_{n-k}$ for $i\neq j$. Furthermore, 
\begin{equation}
D_{n-k}=\bigcup_{i=1}^{s_k} \textnormal{Con}_{\mathcal S_{n-k}}(\sigma_{ki}).\label{equation_derangement}
\end{equation}

 Note that $\chi_{\lambda}(\sigma)=\chi_{\lambda}(\beta)$ for all $\sigma\in \textnormal{Con}_{\mathcal S_n}(\beta)$.  For any $\beta \in {\mathcal S_n}$, let $\varphi(\beta)$ denote the partition of $n$ induced by the cycle structure of $\beta$. Let $\textnormal{Con}_{\mathcal S_n}(\beta)$ be represented by the partition $\varphi(\beta)\vdash n$.  Then by Theorem \ref{frobenius-schur} and Corollary \ref{cor_cayley_integer}, the eigenvalues of  $\mathcal{F}(n,k)$ are integers given by
  \begin{align}
    \eta_{\lambda} (k) &= \frac{1}{f^\lambda} \sum_{i=1}^{s_k} N_{\mathcal S_{n}}(\varphi(\overline \sigma_{ki}))\chi_\lambda( \varphi(\overline\sigma_{ki})), \label{(2)}
  \end{align}
where  $\chi_\lambda(\varphi(\overline\sigma_{ki}))=\chi_{\lambda}(\overline\sigma_{ki})$.

Assume that  $0<k<n$. Note that each  $\overline\sigma_{ki}$  ($1\leq i\leq s_k$) must consist of at least one $1$-cycle in its cycle decomposition. Therefore  $\varphi(\overline\sigma_{ki})=(\nu_{1},\nu_{2},\dots, \nu_{r}) \vdash n$ and  $\nu_{r}=1$. 
 Note that
$\varphi(\overline\sigma_{ki})-\w l_{\varphi(\overline\sigma_{ki})}=(\nu_{1},\nu_{2},\dots, \nu_{r-1})\vdash (n-1)$.
  We are now ready to state the following lemma which is a special case of  \cite[Theorem 3.4]{rec}.

\begin{lem}\label{Foulkes}   If the Ferrers diagrams obtained from $\lambda$ by removing $1$ node from the right hand side from any row of the diagram so that the resulting diagram will still be a partition of $(n-1)$ are those of $\mu_1, \ldots, \mu_q$, then
  \[\chi_\lambda(\varphi(\overline\sigma_{ki}))=  \sum_{j=1}^q \chi_{\mu_j} (\varphi(\overline\sigma_{ki})-\w l_{\varphi(\overline\sigma_{ki})}), \]
  for all $1\leq i\leq s_k$.
  \end{lem}

\begin{eg}
  Let $n=7$ and $\lambda =(3,3,1)$, then
  \begin{align*}
    \chi_{(3,3,1)} ((6,1)) &= \chi_{(3,3)} ((6)) + \chi_{(3,2,1)} ((6)), \\
    \chi_{(3,3,1)} ((4,2,1)) &= \chi_{(3,3)} ((4,2)) + \chi_{(3,2,1)} ((4,2)), \\
    \chi_{(3,3,1)} ((3,3,1)) &= \chi_{(3,3)} ((3,3)) + \chi_{(3,2,1)} ((3,3)), \\
    \chi_{(3,3,1)} ((2,2,2,1)) &= \chi_{(3,3)} ((2,2,2)) + \chi_{(3,2,1)} ((2,2,2)).
  \end{align*}
\end{eg}

We shall need the following lemma \cite[(7.18) on p. 299]{stanley}.

\begin{lem} \label{size} 
  Let $\lambda=(n^{m_n},\ldots,2^{m_2},1^{m_1})\vdash n$ and $z_\lambda =\prod_{j=1}^n (j^{m_j} m_j!)$, then the size of the conjugacy class represented by $\lambda$ is
  \[ N_{\mathcal S_n}(\lambda) = \frac{n!}{z_\lambda}. \]
\end{lem}

\begin{lem} \label{prepre}
  Let $\lambda=(\lambda_1,\dots, \lambda_r) \vdash (n-k)$ be a derangement, i.e., $\lambda_r\geq 2$. If  
  \[ \nu = ( \lambda, 1^k)\vdash n,\ \textnormal{and}\ \mu = (\lambda, 1^{k-1})\vdash (n-1), \]
  then
  \[ N_{\mathcal S_n}(\nu) = \frac{n}{k} N_{\mathcal S_{n-1}}(\mu). \]
\end{lem}
\begin{proof} The lemma follows from Lemma \ref{size}, by noting that 
  \[ N_{\mathcal S_{n}}(\nu) = \frac{n!}{z_\lambda \times 1 \cdot k!}, \ \textnormal{and}\  N_{\mathcal S_{n-1}}(\mu) = \frac{(n-1)!}{z_\lambda \times 1 \cdot (k-1)!}. \]
  \end{proof}

\begin{thm} \label{eigen-n1} Let $0<k<n$ and $\lambda \vdash n$. If the Ferrers diagrams obtained from $\lambda$ by removing $1$ node from the right hand side from any row of the diagram so that the resulting diagram will still be a partition of $(n-1)$ are those of $\mu_1, \ldots, \mu_q$, then
  \begin{align}
  \eta_{\lambda} (k) &= \frac{n}{k f^\lambda} \sum_{j=1}^q f^{\mu_j} \eta_{\mu_j} (k-1).\notag
  \end{align}
\end{thm}

\begin{proof}  Suppose $k=1$. By equation (\ref{(2)}),
  \begin{align*}
        \eta_{\lambda} (1) &= \frac{1}{f^\lambda} \sum_{i=1}^{s_1} N_{\mathcal S_{n}}(\varphi(\overline\sigma_{1i}))\chi_\lambda(\varphi(\overline\sigma_{1i})).
  \end{align*}
   Note that  $\overline\sigma_{1i}$ consists of exactly one $1$-cycle and $\varphi(\overline\sigma_{1i})=(\nu_{1},\nu_{2},\dots, \nu_{r})\vdash n$ with $\nu_r=1$, $\nu_{r-1}\geq 2$. Therefore $\varphi(\overline\sigma_{1i})-\w l_{\varphi(\overline\sigma_{1i})}=(\nu_{1},\nu_{2},\dots, \nu_{r-1})\vdash (n-1)$
   is a derangement. In fact, $\varphi(\overline\sigma_{1i})-\w l_{\varphi(\overline\sigma_{1i})}$ is the partition of $(n-1)$ that represents $\textnormal{Con}_{\mathcal S_{n-1}}(\sigma_{1i})$.    By Lemma \ref{Foulkes} and Lemma \ref{prepre},
   \begin{align*}
  \eta_{\lambda} (1) &= \frac{1}{f^\lambda} \sum_{i=1}^{s_1} N_{\mathcal S_{n}}(\varphi(\overline\sigma_{1i}))\left ( \sum_{j=1}^q \chi_{\mu_j} (\varphi(\overline\sigma_{1i})-\w l_{\varphi(\overline\sigma_{1i})})  \right) \\
  &= \frac{1}{f^\lambda} \sum_{i=1}^{s_1} nN_{\mathcal S_{n-1}}(\varphi(\overline\sigma_{1i})-\w l_{\varphi(\overline\sigma_{1i})})\left ( \sum_{j=1}^q \chi_{\mu_j} (\varphi(\overline\sigma_{1i})-\w l_{\varphi(\overline\sigma_{1i})})  \right) \\
  &= \frac{n}{f^\lambda} \sum_{j=1}^q \left(\sum_{i=1}^{s_1} N_{\mathcal S_{n-1}}(\varphi(\overline\sigma_{1i})-\w l_{\varphi(\overline\sigma_{1i})}) \chi_{\mu_j} (\varphi(\overline\sigma_{1i})-\w l_{\varphi(\overline\sigma_{1i})}) \right)  \\
  &= \frac{n}{f^\lambda} \sum_{j=1}^q f^{\mu_j} \eta_{\mu_j} (0),
 \end{align*}
where the last equality follows from equations (\ref{equation_derangement}) and (\ref{(2)}). Thus, the theorem holds for $k=1$. 

Suppose $k>1$. (We note here that the proof for $k>1$ is similar to the proof for $k=1$. The reason we distinguish them is to make the proof  easier to comprehend.)

 By equation (\ref{(2)}),
  \begin{align*}
    \eta_{\lambda} (k) &= \frac{1}{f^\lambda} \sum_{i=1}^{s_k} N_{\mathcal S_{n}}(\varphi(\overline \sigma_{ki}))\chi_\lambda( \varphi(\overline\sigma_{ki})).
  \end{align*}
Note that $\overline\sigma_{ki}$ consists of exactly k's $1$-cycle  and $\varphi(\overline\sigma_{ki})=(\nu_{1},\nu_{2},\dots, \nu_{r})\vdash n$ with $\nu_j=1$ for $r-k+1\leq j\leq r$ and $\nu_{r-k}\geq 2$. Let $\overline {\overline\sigma}_{ki}$ be the extension of $\sigma_{ki}$ to $\mathcal S_{n-1}$, i.e., $\overline {\overline\sigma}_{ki}(j)=\sigma_{ki}(j)$ for $1\leq j\leq n-k$ and $\overline {\overline\sigma}_{ki}(j)=j$ for $n-k+1\leq j\leq n-1$.  Note that $\varphi(\overline\sigma_{ki})-\w l_{\varphi(\overline\sigma_{ki})}=(\nu_{1},\nu_{2},\dots, \nu_{r-1})\vdash (n-1)$
is the partition of $(n-1)$ that represents $\textnormal{Con}_{\mathcal S_{n-1}}(\overline {\overline\sigma}_{ki})$. Furthermore,
\begin{equation}
\mathcal{S}(n-1,k-1)=\bigcup_{i=1}^{s_k} \textnormal{Con}_{\mathcal S_{n-1}}(\overline{\overline \sigma}_{ki}).\notag
\end{equation}
 Therefore, by Theorem \ref{frobenius-schur}, 
  \begin{align}
    \eta_{\mu_j} (k-1) &= \frac{1}{f^{\mu_j}} \sum_{i=1}^{s_k} N_{\mathcal S_{n-1}}(\varphi(\overline\sigma_{ki})-\w l_{\varphi(\overline\sigma_{ki})})\chi_{\mu_j}(\varphi(\overline\sigma_{ki})-\w l_{\varphi(\overline\sigma_{ki})}). \notag
  \end{align}
 By Lemma \ref{Foulkes} and Lemma \ref{prepre},
   \begin{align*}
  \eta_{\lambda} (k) &= \frac{1}{f^\lambda} \sum_{i=1}^{s_k} N_{\mathcal S_{n}}(\varphi(\overline\sigma_{ki}))\left ( \sum_{j=1}^q \chi_{\mu_j} (\varphi(\overline\sigma_{ki})-\w l_{\varphi(\overline\sigma_{ki})})  \right) \\
  &= \frac{1}{f^\lambda} \sum_{i=1}^{s_k} \frac{n}{k}N_{\mathcal S_{n-1}}(\varphi(\overline\sigma_{ki})-\w l_{\varphi(\overline\sigma_{ki})})\left ( \sum_{j=1}^q \chi_{\mu_j} (\varphi(\overline\sigma_{ki})-\w l_{\varphi(\overline\sigma_{ki})})  \right) \\
  &= \frac{n}{kf^\lambda} \sum_{j=1}^q \left(\sum_{i=1}^{s_k} N_{\mathcal S_{n-1}}(\varphi(\overline\sigma_{ki})-\w l_{\varphi(\overline\sigma_{ki})}) \chi_{\mu_j} (\varphi(\overline\sigma_{1i})-\w l_{\varphi(\overline\sigma_{1i})}) \right)  \\
  &= \frac{n}{kf^\lambda} \sum_{j=1}^q f^{\mu_j} \eta_{\mu_j} (k-1).
 \end{align*}
Hence, the theorem holds for $k>1$. 
\end{proof}

\newpage
\section{Inequalities for the eigenvalues of $\mathcal{F}(n,0)$}

For convenience, if $\lambda=(n)$, we set
\begin{equation}
d_n=\eta_\lambda (0).\notag
\end{equation}
By Theorem \ref{ku-wongformula}, 
\begin{equation}
d_n=(-1)^n+nd_{n-1},\ \ \textnormal{for $n\geq 1$,}\label{equation_d}
\end{equation} 
where $d_0=1$. Note that $d_1=0$ and $d_n> 0$ for all $n\neq 1$. Furthermore, for $n\geq 3$,
\begin{align}
d_n&=(-1)^n+nd_{n-1}\notag\\
&\geq nd_{n-1} -1\label{equation_d1}\\
&= (n-1)d_{n-1}+d_{n-1}-1\geq (n-1)d_{n-1}.\label{equation_d2}
\end{align}

\begin{lem} \label{res1}
  Let  $1 \leq p \leq n-1$. If $\lambda = (n-p,1^p)$ and $\mu = (n-p+1,1^{p-1})$ are partitions of $[n]$, then
  \[ f^\lambda \left|\eta_\lambda (0) \right| \leq f^\mu \left|\eta_\mu (0) \right|. \]
  Furthermore, equality holds if and only if $p=1$ or $n-p=1$.
\end{lem}
\begin{proof} Note that
  \[ f^\lambda = \frac{n!}{H^\lambda} = \frac{n!}{n (n-p-1)! p!} \quad \text{and} \quad f^\mu = \frac{n!}{H^\mu} = \frac{n!}{n (n-p)! (p-1)!}. \]
 By Theorem \ref{rentelnformula} and equation (\ref{equation_d}),
  \begin{align*}
     |\eta_\lambda(0)| &= \left| 1 + (-1)^{n-p} n d_{n-p-1} \right|, \\
     |\eta_\mu(0)| &= \left| 1+(-1)^{n-p+1} n d_{n-p} \right|\\
     &= \left| 1-n+(-1)^{n-p+1} n(n-p) d_{n-p-1} \right|.
  \end{align*}
  Therefore, it is sufficient to show that
  \begin{align}
     P_L=(n-p) \left| 1 + (-1)^{n-p} n d_{n-p-1} \right| & \leq p\left|1-n+(-1)^{n-p+1} n(n-p) d_{n-p-1} \right|=P_R. \notag
   \end{align}
  
\noindent
{\bf Case 1.} Suppose $n$ and $p$ are of  same parity (both even or both odd). Then
  \begin{align*}
     P_R -P_L&= p( n(n-p) d_{n-p-1}+n-1)- (n-p) (1 + n d_{n-p-1}) \\
     &= n(n-p)(p-1)d_{n-p-1}+(p-1)n\geq 0.
  \end{align*}
Note that $P_R -P_L=0$ if and only if $p=1$.
  
  \vskip 0.5cm
  \noindent
{\bf Case 2.} Suppose $n$ and $p$ are of different parity (one even and one odd). Then $d_{n-p-1}\neq 0$, for $n-p\neq 2$. Therefore 
   \begin{align*}
     P_R -P_L&= p(1-n+n(n-p) d_{n-p-1} )- (n-p) ( n d_{n-p-1}-1) \\
     &= n(n-p)(p-1)d_{n-p-1}-(p-1)n\\
      &= n(p-1)((n-p)d_{n-p-1}-1)\geq 0.
  \end{align*}
  Note that $P_R -P_L=0$ if and only if $p=1$ or $n-p=1$.
  \end{proof}

\begin{lem} \label{case_k=1}
 Let $m\geq q\geq 1$ and $n=m+q$. If $\lambda = (m,q)$ and $\mu = (m+1,q-1)$ are partitions of $[n]$, then
 \begin{equation}
  (m-q+1)\left|\eta_\lambda (0) \right| \leq \left|\eta_\mu (0) \right|.\notag 
    \end{equation}
    Furthermore, equality holds if and only if $q=1$ or $m=q=2$.
       \end{lem}
\begin{proof} We shall prove by induction on $q$. Suppose $q=1$. By Corollary \ref{cor_ku-wongformula},  $\vert\eta_\lambda (0)\vert   =d_m+d_{m-1}$.
By equation (\ref{equation_d}), $\left|\eta_\mu (0) \right| =d_{m+1}=(-1)^{m+1}+(m+1)d_m$. Therefore
\begin{align*}
\left|\eta_\mu (0) \right|-m\left|\eta_\lambda (0) \right|  & =(-1)^{m+1}+d_m-md_{m-1}\\
& =(-1)^{m+1}+(-1)^m+md_{m-1}-md_{m-1}=0.
\end{align*}

 Suppose $q\geq 2$. Assume that the lemma holds for $q-1$. By Theorem \ref{rentelnformula},
 \begin{equation}
 \eta_\lambda (0) =(-1)^{m+1}d_{q-1}-(m+1)\eta_{(m-1,q-1)} (0).\notag
 \end{equation}
 By Theorem \ref{ASP_0}, $\sgn(\eta_{\lambda}(0)) =(-1)^q$ and $\sgn((m-1,q-1)) =(-1)^{q-1}$. Thus,
 \begin{equation}
 \vert\eta_\lambda (0)\vert =(-1)^{m-q+1}d_{q-1}+(m+1)\vert \eta_{(m-1,q-1)} (0)\vert.\notag
 \end{equation}
 Similarly, by  Theorem \ref{rentelnformula} and \ref{ASP_0},
 \begin{equation}
 \vert\eta_\mu (0)\vert =(-1)^{m-q+1}d_{q-2}+(m+2)\vert \eta_{(m,q-2)} (0)\vert.\notag
 \end{equation}
By induction,  $(m-q+1)\vert \eta_{(m-1,q-1)} (0)\vert \leq \vert \eta_{(m,q-2)} (0)\vert$.

 Therefore 
 \begin{align*}
&  \vert\eta_\mu (0)\vert-(m-q+1)\vert\eta_\lambda (0)\vert\\
&\hskip 0.5cm \geq (-1)^{m-q}(m-q+1)d_{q-1}+(-1)^{m-q+1}d_{q-2}+\vert \eta_{(m,q-2)} (0)\vert.
 \end{align*}
 If $q=2$, then $d_{q-1}=0$ and $\vert\eta_\mu (0)\vert-(m-q+1)\vert\eta_\lambda (0)\vert\geq d_m+(-1)^{m-1}\geq 0$. Furthermore, equality holds throughout if and only if $m=q=2$.

 Suppose $q\geq 3$. By Corollary \ref{cor_ku-wongformula}, $\vert \eta_{(m,q-2)} (0)\vert=d_m+(q-2)\vert \eta_{(m-1,q-3)} (0)\vert>d_m$, where the last inequality follows from 
$\vert \eta_{(m-1,q-3)} (0)\vert\neq 0$.  If $m\equiv q\mod 2$, then $ \vert\eta_\mu (0)\vert-(m-q+1)\vert\eta_\lambda (0)\vert>(m-q)d_{q-1}+(d_{q-1}-d_{q-2})+d_m>0$.
If $m\not\equiv q\mod 2$, then 
\begin{align*}
\vert\eta_\mu (0)\vert-(m-q+1)\vert\eta_\lambda (0)\vert&>-(m-q+1)d_{q-1}+d_{q-2}+d_m\\
&\geq d_m-(m-q+1)d_{q-1}\\
&\geq (m-1)d_{m-1}-(m-q+1)d_{q-1}\qquad\textnormal{(by equation (\ref{equation_d2}))}\\
&\geq (q-2)d_{q-1}>0.
\end{align*}
 
 This complete the proof of the lemma.
 \end{proof}

\begin{lem}\label{lm_first} If  $m>q\geq 1$ and $k\geq t\geq 1$, then
\begin{equation}
(m-q+k+1)\vert\eta_{(q,q^t)}(0)\vert\leq k\vert\eta_{(m+1,q^t)}(0)\vert.\notag
\end{equation}
Furthermore, equality holds if and only if $q=1$, $m=2$ and $k=t$.
\end{lem}

\begin{proof} We shall prove by induction on $q$. Suppose $q=1$. Then by Corollary \ref{cor_ku-wongformula},
$\vert\eta_{(q,q^t)}(0)\vert=t$ and $\vert\eta_{(m+1,q^t)}(0)\vert=td_m+d_{m+1}$. Note that $m\geq 2$. If $m=2$, then 
\begin{align*}
k\vert\eta_{(m+1,q^t)}(0)\vert-(m-q+k+1)\vert\eta_{(q,q^t)}(0)\vert
&=k(t+2)-(k+2)t\\
&=2(k-t)\geq 0.
\end{align*}
Furthermore, equality holds if and only if $k=t$.

If $m= 3$, then 
\begin{align*}
k\vert\eta_{(m+1,q^t)}(0)\vert-(m-q+k+1)\vert\eta_{(q,q^t)}(0)\vert
&=k(2t+9)-(k+3)t\\
&=kt+3(3k-t)> 0.
\end{align*}
Suppose $m\geq 4$. By equation (\ref{equation_d2}), $d_m\geq (m-1)(m-2)d_{m-2}\geq (m-1)(m-2)$. Since
\begin{align}
k(m-1)(m-2)-(m-q+k+1) &=km^2-(3k+1)m+k\notag\\
&\geq 4km-(3k+1)m+k\notag\\
& =(k-1)m+k>0\notag,
\end{align}
we have $k\vert\eta_{(m+1,q^t)}(0)\vert-(m-q+k+1)\vert\eta_{(q,q^t)}(0)\vert\geq t((k-1)m+k)+d_{m+1}>0$.

Suppose $q\geq 2$. Assume that 
\begin{equation}
(m'-(q-1)+k+1)\vert\eta_{(q-1,(q-1)^t)}(0)\vert\leq k\vert\eta_{(m'+1,(q-1)^t)}(0)\vert,\notag
\end{equation}
for all $m'> q-1$ and $k\geq t\geq 1$.

By Corollary \ref{cor_ku-wongformula},
\begin{align*}
\vert\eta_{(m+1,q^t)}(0)\vert &= q\vert\eta_{(m,(q-1)^t)}(0)\vert+\vert\eta_{(m+1,q^{t-1})}(0)\vert\\
&= q\vert\eta_{(m,(q-1)^t)}(0)\vert+q\vert\eta_{(m,(q-1)^{t-1})}(0)\vert+\vert\eta_{(m+1,q^{t-2})}(0)\vert\\
&\hskip 1cm \vdots\\
&= q\left(\sum_{j=1}^t \vert\eta_{(m,(q-1)^j)}(0)\vert\right)+d_{m+1}.
\end{align*}
Similarly,
\begin{align*}
\vert\eta_{(q,q^t)}(0)\vert &= q\left(\sum_{j=1}^t \vert\eta_{(q-1,(q-1)^j)}(0)\vert\right)+d_{q}.
\end{align*}
By induction, for $1\leq j\leq t$,
\begin{equation}
(m-q+k+1)\vert\eta_{(q-1,(q-1)^j)}(0)\vert\leq k\vert\eta_{(m,(q-1)^j)}(0)\vert.\notag
\end{equation}
By equation (\ref{equation_d2}), $d_{m+1}\geq m(m-1)d_{m-1}\geq m(m-1)d_q$. Note that $m\geq 3$ and
\begin{align*}
km(m-1)-(m-q+k+1) & =km^2-(k+1)m+q-k-1\\
& \geq 3km-(k+1)m+q-k-1\\
& = (2k-1)m+q-k-1\\
& \geq  3(2k-1)+q-k-1\\
& =  5k-4+q>0\\
\end{align*}
Hence, $(m-q+k+1)\vert\eta_{(q,q^t)}(0)\vert\leq k\vert\eta_{(m+1,q^t)}(0)\vert$.

This complete the proof of the lemma.
\end{proof}

\begin{lem}\label{lm_second} If $q\geq 1$ and $t\geq 1$, then
\begin{equation}
\vert\eta_{(q^t,q-1)}(0)\vert< \vert\eta_{(q^t,q)}(0)\vert.\notag
\end{equation}
\end{lem}

\begin{proof} We shall prove by induction on $q$. Suppose $q=1$. Then by Corollary \ref{cor_ku-wongformula},
$\vert\eta_{(q^t,q)}(0)\vert=t>t-1=\vert\eta_{(q^t,q-1)}(0)\vert$. 

Suppose $q\geq 2$. Assume that  the lemma holds for $q-1$.
By Corollary \ref{cor_ku-wongformula},
\begin{align*}
\vert\eta_{(q^t,q)}(0)\vert&= q\vert\eta_{((q-1)^t,q-1)}(0)\vert+\vert\eta_{(q^{t})}(0)\vert\\
\\
\vert\eta_{(q^t,q-1)}(0)\vert &= (q-1)\vert\eta_{((q-1)^t,q-2)}(0)\vert+\vert\eta_{(q^{t})}(0)\vert.
\end{align*}
By induction, $\vert\eta_{((q-1)^t,q-2)}(0)\vert<\vert\eta_{((q-1)^t,q-1)}(0)\vert$.
Hence, $\vert\eta_{(q^t,q-1)}(0)\vert< \vert\eta_{(q^t,q)}(0)\vert$.

This complete the proof of the lemma.
\end{proof}

\begin{lem}\label{lm_case_k>2} Let $m\geq q\geq 1$, $k\geq 2$ and $n=m+kq$. If $\lambda=(m,q^{k-1},q)$ and $\mu=(m+1,q^{k-1},q-1)$ are partitions of $[n]$, then
\begin{equation}
(m-q+1)\vert\eta_{\lambda}(0)\vert\leq k\vert\eta_{\mu}(0)\vert.\notag
\end{equation}
Furthermore, equality holds if and only if $q=1=m$.
\end{lem}

\begin{proof} We shall prove by induction on $q$. Suppose $q=1$. By Corollary \ref{cor_ku-wongformula},
$\vert\eta_{\lambda}(0)\vert=kd_{m-1}+d_m$ and $\vert\eta_{\mu}(0)\vert=(k-1)d_m+d_{m+1}$. If $m=1$, then $\vert\eta_{\lambda}(0)\vert=k=k\vert\eta_{\mu}(0)\vert$ and the lemma holds. If $m=2$, then $2\vert\eta_{\lambda}(0)\vert=2<k(k+1)=k\vert\eta_{\mu}(0)\vert$ and the lemma holds.
Suppose $m\geq 3$. Then by equation (\ref{equation_d2}),
\begin{align*}
& k\vert\eta_{\mu}(0)\vert-(m-q+1)\vert\eta_{\lambda}(0)\vert \\
&=k((k-1)d_m+d_{m+1})-m\left(  kd_{m-1}+d_m \right)\\
&\geq k((k-1)d_m+md_m)-m\left(  kd_{m-1}+d_m \right)\\
&= (k^2+(k-1)m-k)d_{m}-kmd_{m-1}\\
&\geq d_{m-1} \left( (k^2+(k-1)m-k)(m-1)-km\right)\\
&\geq d_{m-1} \left( 2(k^2+(k-1)m-k)-km\right)\\
&=d_{m-1} \left( 2k(k-1)+(k-2)m\right)>0.\\
\end{align*}

Suppose $q\geq 2$. Assume that the lemma holds for $q-1$.
By Theorem \ref{rentelnformula},
\begin{equation}
\eta_{\lambda}(0)=(-1)^k (m+k)\eta_{(m-1,(q-1)^{k-1},q-1)}(0)+(-1)^{m+k}\eta_{((q-1)^{k-1},q-1)}(0).\notag
\end{equation}
By Theorem \ref{ASP_0}, $\sgn(\eta_{\lambda}(0))=(-1)^{kq}$, $\sgn(\eta_{(m-1,(q-1)^{k-1},q-1)}(0))=(-1)^{k(q-1)}$ and \newline
 $\sgn(\eta_{((q-1)^{k-1},q-1)}(0))=(-1)^{(k-1)(q-1)}$. Therefore,
 \begin{equation}
 \vert\eta_{\lambda}(0)\vert=(m+k)\vert\eta_{(m-1,(q-1)^{k-1},q-1)}(0)\vert+(-1)^{m-q+1}\vert\eta_{((q-1)^{k-1},q-1)}(0)\vert.\notag
 \end{equation}
Similarly, 
 \begin{equation}
 \vert\eta_{\mu}(0)\vert=(m+k+1)\vert\eta_{(m,(q-1)^{k-1},q-2)}(0)\vert+(-1)^{m-q}\vert\eta_{((q-1)^{k-1},q-2)}(0)\vert.\notag
 \end{equation}
 By induction,
 \begin{equation}
 (m-q+1)\vert\eta_{(m-1,(q-1)^{k-1},q-1)}(0)\vert\leq k\vert\eta_{(m,(q-1)^{k-1},q-2)}(0)\vert.\notag
 \end{equation}
 
 Suppose $m=q$. Then 
 \begin{align*}
(m-q+1) \vert\eta_{\lambda}(0)\vert & =(m-q+1)\left((m+k)\vert\eta_{(m-1,(q-1)^{k-1},q-1)}(0)\vert-\vert\eta_{((q-1)^{k-1},q-1)}(0)\vert\right)\\
&<(m+k)\left((m-q+1)\vert\eta_{(m-1,(q-1)^{k-1},q-1)}(0)\vert\right)\\
& <(m+k+1)\left(k\vert\eta_{(m,(q-1)^{k-1},q-2)}(0)\vert\right)\\
&\leq k\left((m+k+1)\vert\eta_{(m,(q-1)^{k-1},q-2)}(0)\vert+\vert\eta_{((q-1)^{k-1},q-2)}(0)\vert\right)\\
&=k\vert\eta_{\mu}(0)\vert.
 \end{align*}

Suppose $m>q$. Note that
\begin{align*}
(m-q+1)\vert\eta_{\lambda}(0)\vert &\leq (m-q+1)\left((m+k)\vert\eta_{(m-1,(q-1)^{k-1},q-1)}(0)\vert+\vert\eta_{((q-1)^{k-1},q-1)}(0)\vert\right)\\
 &\leq (m+k)\left(k\vert\eta_{(m,(q-1)^{k-1},q-2)}(0)\vert\right)+(m-q+1)\vert\eta_{((q-1)^{k-1},q-1)}(0)\vert.
\end{align*}
By Lemma \ref{lm_second},
 \begin{align*}
 k\vert\eta_{\mu}(0)\vert &\geq k\left((m+k+1)\vert\eta_{(m,(q-1)^{k-1},q-2)}(0)\vert-\vert\eta_{((q-1)^{k-1},q-2)}(0)\vert\right)\\
 &> k\left((m+k+1)\vert\eta_{(m,(q-1)^{k-1},q-2)}(0)\vert-\vert\eta_{((q-1)^{k-1},q-1)}(0)\vert\right). 
 \end{align*}
Therefore,
\begin{align*}
&k\vert\eta_{\mu}(0)\vert-(m-q+1)\vert\eta_{\lambda}(0)\vert\\
&\geq k\vert\eta_{(m,(q-1)^{k-1},q-2)}(0)\vert-(m-q+k+1)\vert\eta_{((q-1)^{k-1},q-1)}(0)\vert.
\end{align*}
If $q=2$, then by Lemma \ref{lm_first}, $k\vert\eta_{\mu}(0)\vert-(m-q+k+1)\vert\eta_{\lambda}(0)\vert\geq 0$. Suppose $q\geq 3$. By Corollary \ref{cor_ku-wongformula},
\begin{align*}
\vert\eta_{(m,(q-1)^{k-1},q-2)}(0)\vert & =(q-2)\vert\eta_{(m-1,(q-2)^{k-1},q-3)}(0)\vert+\vert\eta_{(m,(q-1)^{k-1})}(0)\vert\\
&>\vert\eta_{(m,(q-1)^{k-1})}(0)\vert.
\end{align*}
It then follows from Lemma \ref{lm_first} that
\begin{align*}
&k\vert\eta_{\mu}(0)\vert-(m-q+1)\vert\eta_{\lambda}(0)\vert\\
&> k\vert\eta_{(m,(q-1)^{k-1})}(0)\vert-(m-q+k+1)\vert\eta_{((q-1)^{k-1},q-1)}(0)\vert>0.
\end{align*}

This complete the proof of the lemma.
\end{proof}

\begin{lem} \label{lemma_special2}    Let $r\geq 0$, $m\geq q\geq 1$, $k\geq 1$, $n=m+kq+\sum_{j=1}^r \alpha_j$, $q>\alpha_1$ and
\begin{align*}
\lambda & = (m,q^{k-1},q,\alpha_1,\dots, \alpha_r),\\
\mu & = (m+1,q^{k-1},q-1,\alpha_1,\dots, \alpha_r),
\end{align*}
be partitions of $[n]$. Then
 \begin{equation}
  (m-q+1)\left|\eta_\lambda (0) \right| \leq k\left|\eta_\mu (0) \right|.\notag 
 \end{equation}
      \end{lem}

\begin{proof} If $r=0$, then the lemma follows from Lemma \ref{case_k=1} or \ref{lm_case_k>2}, depending on whether $k=1$ or $k\geq 2$. Suppose $r\geq 1$. Then $q\geq 2$, for $q>\alpha_1\geq 1$. 
We shall prove by induction on $\alpha_1$. 

Suppose $\alpha_1=1$. Then  $\alpha_1=\cdots=\alpha_r=1$. By Corollary \ref{cor_ku-wongformula},
\begin{align*}
\vert  \eta_\lambda (0) \vert  & = \vert \eta_{(m,q^{k-1},q,\alpha_1,\dots, \alpha_{r-1})} (0) \vert + \vert  \eta_{(m-1,(q-1)^{k-1},q-1)} (0)\vert\\
  & = \vert \eta_{(m,q^{k-1},q,\alpha_1,\dots, \alpha_{r-2})} (0) \vert + 2\vert  \eta_{(m-1,(q-1)^{k-1},q-1)} (0)\vert\\
  &\hskip 1cm\vdots\\
& = \vert \eta_{(m,q^{k-1},q)} (0) \vert + r\vert  \eta_{(m-1,(q-1)^{k-1},q-1)} (0)\vert.
\end{align*}
Similarly, 
\begin{equation}
\vert \eta_\mu (0) \vert= \vert \eta_{(m+1,q^{k-1},q-1)} (0) \vert + r\vert  \eta_{(m,(q-1)^{k-1},q-2)} (0)\vert.\notag
\end{equation}
By  Lemma \ref{case_k=1} or \ref{lm_case_k>2},
\begin{align*}
(m-q+1)\vert \eta_{(m,q^{k-1},q)} (0) \vert & \leq k\vert \eta_{(m+1,q^{k-1},q-1)} (0) \vert,\qquad\textnormal{and}\\
(m-q+1)\vert  \eta_{(m-1,(q-1)^{k-1},q-1)} (0)\vert &\leq k\vert  \eta_{(m,(q-1)^{k-1},q-2)} (0)\vert.
\end{align*}
Hence, $(m-q+1)\left|\eta_\lambda (0) \right| \leq k\left|\eta_\mu (0) \right|$.

Suppose $\alpha_1\geq 2$. Assume that the lemma holds for $\alpha_1-1$.  By Corollary \ref{cor_ku-wongformula},
\begin{align*}
\vert  \eta_\lambda (0) \vert & = \vert \eta_{(m,q^{k-1},q,\alpha_1,\dots, \alpha_{r-1})} (0) \vert + \alpha_r\vert  \eta_{(m-1,(q-1)^{k-1},q-1,\alpha_1-1,\dots,\alpha_r-1)} (0)\vert\\
    &\hskip 1cm\vdots\\
& = \vert \eta_{(m,q^{k-1},q)} (0) \vert + \sum_{j=1}^r \alpha_j\vert  \eta_{(m-1,(q-1)^{k-1},q-1,\alpha_1-1,\dots, \alpha_{j}-1)} (0)\vert.
\end{align*}
Similarly, 
\begin{equation}
\vert \eta_\mu (0) \vert= \vert \eta_{(m+1,q^{k-1},q-1)} (0) \vert + \sum_{j=1}^r \alpha_j\vert  \eta_{(m,(q-1)^{k-1},q-2,\alpha_1-1,\dots, \alpha_{j}-1)} (0)\vert.\notag
\end{equation}
By Lemma \ref{case_k=1} or \ref{lm_case_k>2},
\begin{align*}
(m-q+1)\vert \eta_{(m,q^{k-1},q)} (0) \vert & \leq k\vert \eta_{(m+1,q^{k-1},q-1)} (0) \vert.
\end{align*}
By induction, for $1\leq j\leq r$,
\begin{align*}
(m-q+1)\vert  \eta_{(m-1,(q-1)^{k-1},q-1,\alpha_1-1,\dots, \alpha_{j}-1)} (0)\vert &\leq k\vert  \eta_{(m,(q-1)^{k-1},q-2,\alpha_1-1,\dots, \alpha_{j}-1)} (0)\vert.
\end{align*}
Hence, $(m-q+1)\left|\eta_\lambda (0) \right| \leq k\left|\eta_\mu (0) \right|$.
\end{proof}

The following lemma is obvious.

\begin{lem} \label{pre_box_ineq} If $u\geq v$, then
\begin{equation}
\left(\frac{u+1}{u}\right)\left(\frac{v-1}{v}\right)<1.\notag
\end{equation}
\end{lem}

 \begin{lem} \label{pre_box1}
  Let $r\geq 0$, $k\geq 1$, $m\geq q\geq 2$, $n=m+kq+\sum_{j=1}^r \alpha_j$, $q>\alpha_1$ and
\begin{align*}
\lambda & = (m,q^{k-1},q,\alpha_1,\dots, \alpha_r),\\
\mu & = (m+1,q^{k-1},q-1,\alpha_1,\dots, \alpha_r),
\end{align*}
be partitions of $[n]$. Then
      \begin{align}
     \frac{f^\lambda}{f^\mu}<\frac{(m-q+1)}{k}. \notag
  \end{align}
\end{lem}

\begin{proof} Note that $h_{\mu}(i,j)=h_{\lambda}(i,j)$ for all $i,j$ except when $i=q$, $j=1$ or $j=k+1$.
Let $c_{i}=h_{\lambda}(i,1)$  and  $d_{i}=h_{\lambda}(i,k+1)$ for $1\leq i\leq q-1$. Note that 
$h_{\mu}(i,1)=c_{i}+1$  and  $h_{\mu}(i,k+1)=d_i-1$ for $1\leq i\leq q-1$, and $h_{\mu}(q,1)=h_{\lambda}(q,1)$. Therefore
\begin{align*}
 \frac{f^\lambda}{f^\mu} &= \frac{H^\mu}{H^\lambda}\\
 &= \frac{\left(\prod_{i=1}^{q-1}(c_i+1)\right)\left(\prod_{i=1}^{q-1}(d_i-1)\right)(m+1-q)!(k-1)!}{\left(\prod_{i=1}^{q-1}c_i\right)\left(\prod_{i=1}^{q-1}d_i\right)(m-q)!k!}\\
 &=\left(  \prod_{i=1}^{q-1}\left(\frac{c_i+1}{c_i}\right)\left(\frac{d_i-1}{d_i}\right) \right) \frac{(m+1-q)}{k}\\
 &<\frac{(m-q+1)}{k},
\end{align*}
where the last inequality follows from $c_i>d_i$ and Lemma \ref{pre_box_ineq}.
\end{proof}

\begin{thm} \label{thm_ineq}
 Let $r\geq 0$, $k\geq 1$, $m\geq q\geq 1$, $n=m+kq+\sum_{j=1}^r \alpha_j$, $q>\alpha_1$ and
\begin{align*}
\lambda & = (m,q^{k-1},q,\alpha_1,\dots, \alpha_r),\\
\mu & = (m+1,q^{k-1},q-1,\alpha_1,\dots, \alpha_r),
\end{align*}
be partitions of $[n]$. Then
  \[ f^\lambda \left|\eta_\lambda (0) \right| \leq f^\mu \left|\eta_\mu (0) \right|. \]
  Furthermore, equality holds if and only if $\lambda=(1,1^{n-1})$ or $\lambda=(n-1,1)$.
\end{thm}

\begin{proof}  Suppose $q=1$. Then $r=0$ and the theorem follows from Lemma \ref{res1}.
Suppose $q\geq 2$. By Lemma \ref{lemma_special2} and  \ref{pre_box1},
\begin{equation}
 \frac{f^\lambda}{f^\mu} \left|\eta_\lambda (0) \right|<\frac{(m-q+1)}{k}\left|\eta_\lambda (0) \right|\leq \left|\eta_\mu (0) \right|. \notag
\end{equation}
This complete the proof of the theorem.
\end{proof}

\section{Proof of Theorem \ref{ASP_1}}

\begin{proof} Suppose the Ferrers diagrams obtained from $\lambda$ by removing $1$ node from the right hand side from any row of the diagram so that the resulting diagram will still be a partition of $(n-1)$ are those of $\mu_1, \ldots, \mu_s$. Then by Theorem \ref{eigen-n1}, 
\begin{align}
  \eta_{\lambda} (1) &= \frac{n}{k f^\lambda} \sum_{j=1}^s f^{\mu_j} \eta_{\mu_j} (0).\notag
  \end{align}

Suppose $r=1$. Then $s=1$ and $\mu_1=(\lambda_1-1)=(n-1)$. Thus, $\eta_{\lambda} (1)=\frac{n}{k f^\lambda}  f^{\mu_1} \eta_{\mu_1} (0)\geq 0$ and with equality if and only if $\mu_1=(1)$, i.e., $\lambda=(2)$.

 Suppose $r\geq 2$. If $\lambda_1=\lambda_2$, then the first part of each $\mu_j$ is $\lambda_1$.  By Theorem \ref{ASP_0}, $\sgn(\eta_{\mu_j}(0))=\left(\sum_{i=2}^r \lambda_i\right)-1=|\lambda|-\lambda_{1}-1$. 
Hence, $\eta_{\lambda} (1)=(-1)^{|\lambda|-\lambda_{1}-1}\frac{n}{k f^\lambda} \sum_{j=1}^s f^{\mu_j} \vert \eta_{\mu_j} (0)\vert$. 
Note that $\eta_{\lambda} (1)=0$ if and only if $s=1$ and $\mu_1=(1)$, i.e., $\lambda=(1,1)$. For other partitions $\lambda$,  $\vert\eta_{\lambda} (1)\vert\neq 0$ and  $\sgn(\eta_{\lambda}(1))=|\lambda|-\lambda_{1}-1$.

Suppose $\lambda_1=m+1>\lambda_2=q$. Note that we may write 
\begin{equation}
\lambda  = (m+1,q^{k-1},q,\alpha_1,\dots, \alpha_r),\notag
\end{equation}
where  $r\geq 0$, $k\geq 1$, $m\geq q\geq 1$, and $q>\alpha_1$. Let
\begin{align*}
\mu_1 & = (m,q^{k-1},q,\alpha_1,\dots, \alpha_r),\\
\mu_2 & = (m+1,q^{k-1},q-1,\alpha_1,\dots, \alpha_r).\notag
\end{align*}
By Theorem \ref{ASP_0}, $\sgn(\eta_{\mu_1}(0))=|\lambda|-\lambda_{1}$ and $\sgn(\eta_{\mu_j}(0))=|\lambda|-\lambda_{1}-1$ for $j\geq 2$.
 This implies that 
\begin{align*}
 \eta_{\lambda} (1) &= (-1)^{|\lambda|-\lambda_{1}-1}\frac{n}{k f^\lambda} \left(f^{\mu_2} \vert \eta_{\mu_2}(0))\vert- f^{\mu_1} \vert \eta_{\mu_1}(0))\vert +\sum_{j=3}^s f^{\mu_j} \vert \eta_{\mu_j} (0)\vert\right).
\end{align*}
By Theorem \ref{thm_ineq}, $f^{\mu_2}\vert \eta_{\mu_2} (0)\vert-f^{\mu_1} \vert\eta_{\mu_1} (0)\vert \geq 0$. Furthermore, equality holds if and only if 
\begin{equation}
\mu_1=(1,1^{n-2})\ \textnormal{ or}\ \mu_1=(n-2,1),\notag
\end{equation}
i.e., $\lambda= (2,1^{n-2})$ or $(n-1,1)$. Note also that when this happens, we have $s=2$ and  $\eta_{\lambda} (1)=0$.
For other partitions $\lambda$, $f^{\mu_2}\vert \eta_{\mu_2} (0)\vert-f^{\mu_1} \vert\eta_{\mu_1} (0)\vert > 0$. Hence, $\vert\eta_{\lambda} (1)\vert\neq 0$ and  $\sgn(\eta_{\lambda}(1))=|\lambda|-\lambda_{1}-1$.

This complete the proof of the theorem.
\end{proof}

\newpage

\section{Eigenvalues Table for  $\mathcal{F} (n,1)$}

  \begin{center}
  $n=3$\\
  \end{center}
  \begin{center}
  \begin{tabular}{|c|c||c|c||c|c|}
    \hline
    $\lambda$ & $\eta_\lambda$ & $\lambda$ & $\eta_\lambda$ & $\lambda$ & $\eta_\lambda$\\
    \hline
    $[3]$ & $3$ & $[2,1]$ & $0$ & $[1^3]$ & $-3$ \\
    \hline
  \end{tabular}
  \end{center}
  \begin{center}
  $n=4$\\
  \end{center}
  \begin{center}
  \begin{tabular}{|c|c||c|c||c|c||c|c||c|c|}
    \hline
    $\lambda$ & $\eta_\lambda$ & $\lambda$ & $\eta_\lambda$ & $\lambda$ & $\eta_\lambda$ & $\lambda$ & $\eta_\lambda$ & $\lambda$ & $\eta_\lambda$ \\
    \hline
    $[4]$ & $8$ & $[3,1]$ & $0$ & $[2,2]$ & $-4$ & $[2,1^2]$ & $0$ & $[1^4]$ & $8$ \\
    \hline
  \end{tabular}
  \end{center}
  \begin{center}
  $n=5$\\
  \end{center}
  \begin{center}
  \begin{tabular}{|c|c||c|c||c|c||c|c|}
    \hline
    $\lambda$ & $\eta_\lambda$ & $\lambda$ & $\eta_\lambda$ & $\lambda$ & $\eta_\lambda$ & $\lambda$ & $\eta_\lambda$ \\
    \hline
    $[5]$ & $45$  & $[3,2]$ & $-3$ & $[2^2,1]$ & $9$ &  $[1^5]$ & $-15$ \\
    $[4,1]$ & $0$ & $[3,1^2]$ & $-5$ & $[2,1^3]$ & $0$ & & \\
    \hline
  \end{tabular}
  \end{center}
  \begin{center}
  $n=6$\\
  \end{center}
  \begin{center}
  \begin{tabular}{|c|c||c|c||c|c||c|c|}
    \hline
    $\lambda$ & $\eta_\lambda$ & $\lambda$ & $\eta_\lambda$ & $\lambda$ & $\eta_\lambda$ & $\lambda$ & $\eta_\lambda$\\
    \hline
    $[6]$ & $264$ & $[4,1^2]$ & $-12$ &$[3,1^3]$ & $12$ & $[2,1^4]$ & $0$\\
    $[5,1]$ & $0$ & $[3^2]$ & $24$ & $[2^3]$ & $-24$ & $[1^6]$ & $24$\\
    $[4,2]$ & $-16$ & $[3,2,1]$ & $9$ & $[2^2,1^2]$ & $-16$ & &\\
    \hline
  \end{tabular}
  \end{center}
  \newpage
  \begin{center}
  $n=7$\\
  \end{center}
  \begin{center}
    \begin{tabular}{|c|c||c|c||c|c||c|c||c|c|}
    \hline
    $\lambda$ & $\eta_\lambda$ & $\lambda$ & $\eta_\lambda$ & $\lambda$ & $\eta_\lambda$ & $\lambda$ & $\eta_\lambda$ & $\lambda$ & $\eta_\lambda$ \\
    \hline
    $[7]$ & $1855$ & $[5,1^2]$ & $-63$ & $[4,1^3]$ & $28$ & $[3,2,1^2]$ & $-17$ & $[2^2,1^3]$ & $25$ \\
    $[6,1]$ & $0$ & $[4,3]$ & $40$ & $[3^2,1]$ & $-45$ & $[3,1^4]$ & $-21$ & $[2,1^5]$ & $0$ \\
    $[5,2]$ & $-65$ & $[4,2,1]$ & $37$ & $[3,2^2]$ & $-15$ & $[2^3,1]$ & $40$ & $[1^7]$ & $-35$    \\
    \hline
  \end{tabular}
  \end{center}
  \begin{center}
  $n=8$\\
  \end{center}
  \begin{center}
  \begin{tabular}{|c|c||c|c||c|c||c|c|}
    \hline
    $\lambda$ & $\eta_\lambda$ & $\lambda$ & $\eta_\lambda$ & $\lambda$ & $\eta_\lambda$ & $\lambda$ & $\eta_\lambda$\\
    \hline
    $[8]$ & $14832$ & $[5,1^3]$ & $144$ & $[3^2,2]$ & $80$ & $[2^3,1^2]$ & $-60$ \\
    $[7,1]$ & $0$ & $[4^2]$ & $-168$ & $[3^2,1^2]$ & $72$ & $[2^2,1^4]$ & $-36$ \\
    $[6,2]$ & $-372$ & $[4,3,1]$ & $-72$ & $[3,2^2,1]$ & $24$ & $[2,1^6]$ & $0$ \\
    $[6,1^2]$ & $-352$ & $[4,2^2]$ & $-72$ & $[3,2,1^3]$ & $27$ & $[1^8]$ & $48$ \\
    $[5,3]$ & $180$ & $[4,2,1^2]$ & $-64$ & $[3,1^5]$ & $32$ & & \\
    $[5,2,1]$ & $153$ & $[4,1^4]$ & $-48$ & $[2^4]$ & $-72$ & & \\
    \hline
  \end{tabular}
  \end{center}
  \begin{center}
  $n=9$\\
  \end{center}
  \begin{center}
  \begin{tabular}{|c|c||c|c||c|c||c|c|}
    \hline
    $\lambda$ & $\eta_\lambda$ & $\lambda$ & $\eta_\lambda$ & $\lambda$ & $\eta_\lambda$ & $\lambda$ & $\eta_\lambda$\\
    \hline
    $[9]$ & $133497$ & $[5,3,1]$ & $-315$ & $[4,2,1^3]$ &$97$ & $[3,1^6]$ & $-45$ \\
    $[8,1]$ & $0$ & $[5,2,2]$ & $-273$ & $[4,1^5]$ & $72$ & $[2^4,1]$ & $105$ \\
    $[7,2]$ & $-2471$ & $[5,2,1^2]$ & $-263$ & $[3^3]$ & $-171$ & $[2^3,1^3]$ & $84$ \\
    $[7,1^2]$ & $-2385$ & $[5,1^4]$ & $-243$ & $[3^2,2,1]$ & $-120$ & $[2^2,1^5]$ & $49$ \\
    $[6,3]$ & $924$ & $[4^2,1]$ & $267$ & $[3^3,1^3]$ & $-105$ & $[2,1^7]$ & $0$ \\
    $[6,2,1]$ & $849$ & $[4,3,2]$ & $120$ & $[3,2^3]$ & $-33$ & $[1^9]$ & $-63$\\
    $[6,1^3]$ & $792$ & $[4,3,1^2]$ & $112$ & $[3,2^2,1^2]$ & $-35$ & & \\
    $[5,4]$ & $-375$ & $[4,2^2,1]$ & $112$ & $[3,2,1^4]$ & $-39$ & & \\
    \hline
  \end{tabular}
  \end{center}
  \begin{center}
  $n=10$\\
  \end{center}
  \begin{center}
  \begin{tabular}{|c|c||c|c||c|c||c|c|}
    \hline
    $\lambda$ & $\eta_\lambda$ & $\lambda$ & $\eta_\lambda$ & $\lambda$ & $\eta_\lambda$ & $\lambda$ & $\eta_\lambda$\\
    \hline
    $[10]$ & $1334960$ & $[6,1^4]$ & $-1320$ & $[4,3,2,1]$ &$-175$ & $[3,2^2,1^3]$ & $48$ \\
    $[9,1]$ & $0$ & $[5^2]$ & $1280$ & $[4,3,1^3]$ & $-160$ & $[3,2,1^5]$ & $53$ \\
    $[8,2]$ & $-19072$ & $[5,4,1]$ & $585$ & $[4,2^3]$ & $-172$ & $[3,1^7]$ & $60$ \\
    $[8,1^2]$ & $-18540$ & $[5,3,2]$ & $504$ & $[4,2^2,1^2]$ & $-160$ & $[2^5]$ & $-160$ \\
    $[7,3]$ & $5936$ & $[5,3,1^2]$ & $480$ & $[4,2,1^4]$ & $-136$ & $[2^4,1^2]$ & $-144$ \\
    $[7,2,1]$ & $5561$ & $[5,2^2,1]$ & $416$ & $[4,1^6]$ & $-100$ & $[2^3,1^4]$ & $-112$\\
    $[7,1^3]$ & $5300$ & $[5,2,1^3]$ & $395$ & $[3^3,1]$ & $248$ & $[2^2,1^6]$ & $-64$\\
    $[6,4]$ & $-1872$ & $[5,1^5]$ & $360$ & $[3^2,2^2]$ & $180$ & $[2,1^8]$ & $0$ \\
    $[6,3,1]$ & $-1584$ & $[4^2,2]$ & $-420$ & $[3^2,2,1^2]$ & $168$ & $[1^{10}]$ & $80$ \\
    $[6,2^2]$ & $-1488$ & $[4^2,1^2]$ & $-388$ & $[3^2,1^4]$ & $144$ & & \\
    $[6,2,1^2]$ & $-1432$ & $[4,3^2]$ & $-184$ & $[3,2^3,1]$ & $45$ & & \\
    \hline
  \end{tabular}
  \end{center}
  \newpage
{\small  
  \begin{center}
  $n=11$, $\lambda_1 \geq 5$\\
  \end{center}
  \begin{center}
  \begin{tabular}{|c|c||c|c||c|c|}
    \hline
    $\lambda$ & $\eta_\lambda$ & $\lambda$ & $\eta_\lambda$ & $\lambda$ & $\eta_\lambda$ \\
    \hline
    $[11]$ & $14684571$ & $[7,2,1^2]$ & $-9269$ & $[5,4,2]$ &$-875$   \\
    $[10,1]$ & $0$ & $[7,1^4]$ & $-8745$ & $[5,4,1^2]$ & $-837$  \\
    $[9,2]$ & $-166869$ & $[6,5]$ & $3576$ & $[5,3^2]$ & $-801$  \\
    $[9,1^2]$ & $-163163$ & $[6,4,1]$ & $2851$  & $[5,3,2,1]$ & $-720$ \\
    $[8,3]$ & $44496$ & $[6,3,2]$ & $2464$ & $[5,3,1^3]$ & $-675$  \\
    $[8,2,1]$ & $42381$ & $[6,3,1^2]$ & $2376$  & $[5,2^3]$ & $-603$ \\
    $[8,1^3]$ & $40788$ & $[6,2^2,1]$ & $2232$ & $[5,2^2,1^2]$ & $-585$ \\
    $[7,4]$ & $-11109$ & $[6,2,1^3]$ & $2121$ & $[5,2,1^4]$ & $-549$ \\
    $[7,3,1]$ & $-10017$ & $[6,1^5]$ & $1936$ & $[5,1^6]$ & $-495$\\
    $[7,2^2]$ & $-9531$ & $[5^2,1]$ & $-1863$ & & \\
    \hline
  \end{tabular}
  \end{center}
  \begin{center}
  $n=12$, $\lambda_1 \geq 6$\\
  \end{center}
  \begin{center}
  \begin{tabular}{|c|c||c|c||c|c|}
    \hline
    $\lambda$ & $\eta_\lambda$ & $\lambda$ & $\eta_\lambda$ & $\lambda$ & $\eta_\lambda$ \\
    \hline
    $[12]$ & $176214840$ & $[8,2,1^2]$ & $-69928$ & $[6,5,1]$ &$-5112$   \\
    $[11,1]$ & $0$ & $[8,1^4]$ & $-66744$ & $[6,4,2]$ & $-4160$  \\
    $[10,2]$ & $-1631620$ & $[7,5]$ & $19880$ & $[6,4,1^2]$ & $-4008$  \\
    $[10,1^2]$ & $-1601952$ & $[7,4,1]$ & $16665$  & $[6,3^2]$ & $-3684$ \\
    $[9,3]$ & $381420$ & $[7,3,2]$ & $15264$ & $[6,3,2,1]$ & $-3465$  \\
    $[9,2,1]$ & $367113$ & $[7,3,1^2]$ & $14840$  & $[6,3,1^3]$ & $-3300$ \\
    $[9,1^3]$ & $355992$ & $[7,2^2,1]$ & $14120$ & $[6,2^3]$ & $-3192$ \\
    $[8,4]$ & $-80112$ & $[7,2,1^3]$ & $13595$ & $[6,2^2,1^2]$ & $-3100$ \\
    $[8,3,1]$ & $-74160$ & $[7,1^5]$ & $12720$ & $[6,2,1^4]$ & $-2916$\\
    $[8,2^2]$ & $-71520$ & $[6^2]$ & $-10860$ & $[6,1^6]$ & $-2640$\\
    \hline
  \end{tabular}
  \end{center}
  
  \begin{center}
  $n=13$, $\lambda_1 \geq 6$\\
  \end{center}
  \begin{center}
  \begin{tabular}{|c|c||c|c||c|c|}
    \hline
    $\lambda$ & $\eta_\lambda$ & $\lambda$ & $\eta_\lambda$ & $\lambda$ & $\eta_\lambda$ \\
    \hline
    $[13]$ & $2290792933$ & $[8,3,1^2]$ & $108768$ & $[6^2,1]$ &$14877$   \\
    $[12,1]$ & $0$ & $[8,2^2,1]$ & $104896$ & $[6,5,2]$ & $7152$  \\
    $[11,2]$ & $-17621483$ & $[8,2,1^3]$ & $101713$ & $[6,5,1^2]$ & $6904$  \\
    $[11,1^2]$ & $-17354493$ & $[8,1^5]$ & $96408$  & $[6,4,3]$ & $5997$ \\
    $[10,3]$ & $3671140$ & $[7,6]$ & $-36155$ & $[6,4,2,1]$ & $5616$  \\
    $[10,2,1]$ & $3559897$ & $[7,5,1]$ & $-27953$  & $[6,4,1^3]$ & $5343$ \\
    $[10,1^3]$ & $3470896$ & $[7,4,2]$ & $-23805$ & $[6,3^2,1]$ & $4972$ \\
    $[9,4]$ & $-667467$ & $[7,4,1^2]$ & $-23147$ & $[6,3,2^2]$ & $4752$ \\
    $[9,3,1]$ & $-629343$ & $[7,3^2]$ & $-22271$ & $[6,3,2,1^2]$ & $4620$\\
    $[9,2^2]$ & $-611853$ & $[7,3,2,1]$ & $-21200$ & $[6,3,1^4]$ & $4356$\\
    $[9,2,1^2]$ & $-600731$ & $[7,3,1^3]$ & $-20405$  & $[6,2^3,1]$ & $4257$ \\
    $[9,1^4]$ & $-578487$ & $[7,2^3]$ & $-19853$ & $[6,2^2,1^3]$ & $4092$ \\
    $[8,5]$ & $133408$ & $[7,2^2,1^2]$ & $-19415$ & $[6,2,1^5]$ & $3817$ \\
    $[8,4,1]$ & $118683$ & $[7,2,1^4]$ & $-18539$ & $[6,1^7]$ & $3432$\\
    $[8,3,2]$ & $111240$ & $[7,1^6]$ & $-17225$ & & \\
    \hline
  \end{tabular}
  \end{center}
  }
  \newpage
  \begin{center}
  $n=15$\\
  \end{center}
  {\small
  \begin{center}
  \begin{tabular}{|c|c||c|c||c|c|}
    \hline
    $\lambda$ & $\eta_\lambda$ & $\lambda$ & $\eta_\lambda$ & $\lambda$ & $\eta_\lambda$ \\
    \hline
    $[15]$ & $481066515735$ & $[8,2,1^5]$ & $177997$ & $[6,3,1^6]$ &$6864$   \\
    $[14,1]$ & $0$ & $[8,1^7]$ & $166860$ & $[6,2^4,1]$ & $7032$  \\
    $[13,2]$ & $-2672591753$ & $[7^2,1]$ & $-133761$ & $[6,2^3,1^3]$ & $6813$  \\
    $[13,1^2]$ & $-2643222615$ & $[7,6,2]$ & $-65079$  & $[6,2^2,1^5]$ & $6448$ \\
    $[12,3]$ & $458158584$ & $[7,6,1^2]$ & $-63305$ & $[6,2,1^7]$ & $5937$  \\
    $[12,2,1]$ & $448546869$ & $[7,5,3]$ & $-52211$  & $[6,1^9]$ & $5280$ \\
    $[12,1^3]$ & $440537100$ & $[7,5,2,1]$ & $-49700$ & $[5^3]$ & $-11235$ \\
    $[11,4]$ & $-66080553$ & $[7,5,1^3]$ & $-47825$ & $[5^2,4,1]$ & $-7880$ \\
    $[11,3,1]$ & $-63633141$ & $[7,4^2]$ & $-46581$ & $[5^2,3,2]$ & $-6803$\\
    $[11,2^2]$ & $-62476167$ & $[7,4,3,1]$ & $-43193$ & $[5^2,3,1^2]$ & $-6545$\\
    $[11,2,1^2]$ & $-61675193$ & $[7,4,2^2]$ & $-41655$  & $[5^2,2^2,1]$ & $-6063$ \\
    $[11,1^4]$ & $-60073245$ & $[7,4,2,1^2]$ & $-40733$ & $[5^2,2,1^3]$ & $-5760$ \\
    $[10,5]$ & $9789640$ & $[7,4,1^4]$ & $-38889$ & $[5^2,1^5]$ & $-5255$ \\
    $[10,4,1]$ & $9154035$ & $[7,3^2,2]$ & $-38981$ & $[5,4^2,2]$ &$-3285$   \\
    $[10,3,2]$ & $8810736$ & $[7,3^2,1^2]$ & $-38111$ & $[5,4^2,1^2]$ & $-3183$  \\
    $[10,3,1^2]$ & $8677240$ & $[7,3,2^2,1]$ & $-36729$ & $[5,4,3^2]$ & $-3111$  \\
    $[10,2^2,1]$ & $8484424$ & $[7,3,2,1^3]$ & $-35616$  & $[5,4,3,2,1]$ & $-2925$ \\
    $[10,2,1^3]$ & $8306425$ & $[7,3,1^5]$ & $-33761$ & $[5,4,3,1^3]$ & $-2793$  \\
    $[10,1^5]$ & $8009760$ & $[7,2^4]$ & $-33965$  & $[5,4,2^3]$ & $-2685$ \\
    $[9,6]$ & $-1668105$ & $[7,2^3,1^2]$ & $-33351$ & $[5,4,2^2,1^2]$ & $-2615$ \\
    $[9,5,1]$ & $-1468563$ & $[7,2^2,1^4]$ & $-32123$ & $[5,4,2,1^4]$ & $-2475$ \\
    $[9,4,2]$ & $-1359655$ & $[7,2,1^6]$ & $-30281$ & $[5,4,1^6]$ & $-2265$\\
    $[9,4,1^2]$ & $-1334937$ & $[7,1^8]$ & $-27825$ & $[5,3^3,1]$ & $-2649$\\
    $[9,3^2]$ & $-1311141$ & $[6^2,3]$ & $27903$  & $[5,3^2,2^2]$ & $-2495$ \\
    $[9,3,2,1]$ & $-1271400$ & $[6^2,2,1]$ & $26064$ & $[5,3^2,2,1^2]$ & $-2421$ \\
    $[9,3,1^3]$ & $-1239615$ & $[6^2,1^3]$ & $24765$ & $[5,3^2,1^4]$ & $-2273$ \\
    $[9,2^3]$ & $-1223703$ & $[6,5,4]$ & $14181$ & $[5,3,2^3,1]$ &$-2160$   \\
    $[9,2^2,1^2]$ & $-1205165$ & $[6,5,3,1]$ & $12832$ & $[5,3,2^2,1^3]$ & $-2079$  \\
    $[9,2,1^4]$ & $-1168089$ & $[6,5,2^2]$ & $12252$ & $[5,3,2,1^5]$ & $-1944$  \\
    $[9,1^6]$ & $-1112475$ & $[6,5,2,1^2]$ & $11920$  & $[5,3,1^7]$ & $-1755$ \\
    $[8,7]$ & $393072$ & $[6,5,1^4]$ & $11256$ & $[5,2^5]$ & $-1725$  \\
    $[8,6,1]$ & $297585$ & $[6,4^2,1]$ & $11352$  & $[5,2^4,1^2]$ & $-1691$ \\
    $[8,5,2]$ & $250140$ & $[6,4,3,2]$ & $10287$ & $[5,2^3,1^4]$ & $-1623$ \\
    $[8,5,1^2]$ & $244588$ & $[6,4,3,1^2]$ & $9997$ & $[5,2^2,1^6]$ & $-1521$ \\
    $[8,4,3]$ & $228897$ & $[6,4,2^2,1]$ & $9515$ & $[5,2,1^8]$ & $-1385$\\
    $[8,4,2,1]$ & $220308$ & $[6,4,2,1^3]$ & $9152$ & $[5,1^{10}]$ & $-1215$\\
    $[8,4,1^3]$ & $213627$ & $[6,4,1^5]$ & $8547$  & $[4^3,3]$ & $3420$ \\
    $[8,3^2,1]$ & $209752$ & $[6,3^3]$ & $8808$ & $[4^3,2,1]$ & $3075$ \\
    $[8,3,2^2]$ & $203940$ & $[6,3^2,2,1]$ & $8409$ & $[4^3,1^3]$ & $2856$ \\
    $[8,3,2,1^2]$ & $200232$ & $[6,3^2,1^3]$ & $8100$ & $[4^2,3^2,1]$ &$2305$   \\
    $[8,3,1^4]$ & $192816$ & $[6,3,2^3]$ & $7920$ & $[4^2,3,2^2]$ & $2143$  \\
    $[8,2^3,1]$ & $190725$ & $[6,3,2^2,1^2]$ & $7744$ & $[4^2,3,2,1^2]$ & $2061$  \\
    $[8,2^2,1^3]$ & $185952$ & $[6,3,2,1^4]$ & $7392$  & $[4^2,3,1^4]$ & $1897$ \\
    \hline
  \end{tabular}
  \end{center}
  }

  \begin{center}
  \begin{tabular}{|c|c||c|c||c|c|}
    \hline
    $\lambda$ & $\eta_\lambda$ & $\lambda$ & $\eta_\lambda$ & $\lambda$ & $\eta_\lambda$ \\
    \hline
    $[4^2,2^3,1]$ & $1788$ & $[4,2^3,1^5]$ & $697$  & $[3,2^6]$ & $-123$  \\
    $[4^2,2^2,1^3]$ & $1695$ & $[4,2^2,1^7]$ & $520$   & $[3,2^5,1^2]$ & $-125$  \\
    $[4^2,2,1^5]$ & $1540$ & $[4,2,1^9]$ & $421$ & $[3,2^4,1^4]$ & $-129$  \\
    $[4^2,1^7]$ & $1323$ &$[4,1^{11}]$ & $300$ &  $[3,2^3,1^6]$ & $-135$ \\
    $[4,3^3,2]$ & $760$ & $[3^5]$ & $-1545$ & $[3,2^2,1^8]$ & $-143$ \\
    $[4,3^3,1^2]$ & $744$ & $[3^4,2,1]$ & $-1368$ & $[3,2,1^9]$ & $-153$ \\
    $[4,3^2,2^2,1]$ & $736$ & $[3^4,1^3]$ & $-1263$ & $[3,1^{12}]$ & $-165$\\
    $[4,3^2,2,1^3]$ & $709$ &$[3^3,2^3]$ & $-1119$ & $[2^7,1]$ & $624$\\
    $[4,3^2,1^5]$ & $664$ & $[3^3,2^2,1^2]$ & $-1073$ & $[2^6,1^3]$ & $585$\\
    $[4,3,2^4]$ & $720$ &  $[3^3,2,1^4]$ & $-981$ & $[2^5,1^5]$ & $520$ \\
    $[4,3,2^3,1^2]$ & $700$ & $[3^3,1^6]$ & $-843$ &  $[2^4,1^7]$ & $429$\\
    $[4,3,2^3,1^4]$ & $660$ & $[3^2,2^4,1]$ & $-693$ &  $[2^3,1^9]$ & $312$ \\
    $[4,3,2,1^6]$ & $600$ & $[3^2,2^3,1^3]$ & $-660$ & $[2^2,1^{11}]$ & $169$  \\
    $[4,3,1^8]$ & $520$ & $[3^2,2^2,1^5]$ & $-605$  &  $[2,1^{13}]$ & $0$ \\
    $[4,2^5,1]$ & $685$ & $[3^2,2,1^7]$ & $-528$ & $[1^{15}]$ & $-195$ \\
    $[4,2^4,1^3]$ & $652$ & $[3^2,1^9]$ & $-429$  & & \\
    \hline
  \end{tabular}
  \end{center}

\end{document}